\documentclass[leqno,12pt]{amsproc}

\usepackage{amsmath, amsthm, amsfonts, amssymb, latexsym, amscd, stackrel,url,hyperref}
\usepackage{enumerate}
\usepackage{color}
 
\usepackage{xy}
\xyoption{all}

\makeatletter
\def\url@smallstyle{%
  \@ifundefined{selectfont}{\def\UrlFont{\sf}}{\def\UrlFont{\small\ttfamily}}}
\makeatother
\numberwithin{equation}{section}

\theoremstyle{plain}
    \newtheorem{theorem}[equation]{Theorem}
    \newtheorem{lemma}[equation]{Lemma}
    
    \newtheorem{proposition}[equation]{Proposition}

    \newtheorem*{theorem*}{Theorem}
    \newtheorem*{proposition*}{Proposition}
    \newtheorem*{corollary*}{Corollary}
    \newtheorem*{lemma*}{Lemma}
    \newtheorem*{conjecture*}{Conjecture}
    \newtheorem{definition-theorem}[equation]{Definition/Theorem}
    \newtheorem{definition-lemma}[equation]{Definition/Lemma}
\theoremstyle{definition}
    \newtheorem{definition}[equation]{Definition}
    \newtheorem{example}[equation]{Example}

    \newtheorem{remark}[equation]{Remark}

    \newcommand{\R}{\mathbb{R}}

    \newcommand{\Z}{\mathbb{Z}}

      \newcommand{\HC}{\mathcal{HC}}

     \renewcommand{\S}{\mathcal{S}}

   	\renewcommand{\phi}{\varphi}
    \let\epsilon\varepsilon

    \newcommand{\Bounded}{\operatorname{B}}
    \newcommand{\Compact}{\operatorname{K}}
    \newcommand{\CB}{{\operatorname{CB}}}

\DeclareMathOperator{\Hom}{Hom}

\DeclareMathOperator{\Ind}{Ind}
\DeclareMathOperator{\Res}{Res}

\DeclareMathOperator{\HMod}{\mathsf{Hilb}}
\DeclareMathOperator{\CMod}{\mathsf{C*Mod}}
\DeclareMathOperator{\OMod}{\mathsf{OpMod}}

\DeclareMathOperator{\SFMod}{\mathsf{SFMod}}

\makeatletter
\newcommand{\eqqed}{\displaymath@qed}
\makeatother

\begin{document}
\title[Parabolic Induction  and Operator Spaces]{Parabolic Induction, Categories of Representations and Operator Spaces}

\author{Tyrone Crisp}
\thanks{Partially supported by the Danish National Research Foundation through the Centre for Symmetry and Deformation (DNRF92).}
\address{Department of Mathematics, University of Copenhagen,
Universitetsparken 5,
2100 Copenhagen \O,  Denmark.}
\email{crisp@math.ku.dk}

\author{Nigel Higson}
\thanks{Partially supported by the US National Science Foundation through the grant DMS-1101382.}
\address{Department of Mathematics, Penn State University, University Park, PA 16802, USA.}
\email{higson@math.psu.edu}


\begin{abstract}
We study some aspects of the functor of parabolic induction within the context of   reduced group $C^*$-algebras and related operator algebras.   We   explain how Frobenius reciprocity  fits naturally within the context of operator modules, and examine the prospects for an operator algebraic formulation of Bernstein's reciprocity theorem (his second adjoint theorem).\end{abstract}


\maketitle

\section{Introduction}
\label{intro_section}
 
Harish-Chandra famously  decomposed the regular representation of a real reductive group $G$ into an explicit integral of its isotypical parts.  His program to do so  had two parts:
\begin{enumerate}[\rm (a)]
\item the classification of  the so-called \emph{cuspidal} representations of $G$,    and of the Levi subgroups of $G$; and 
\item   the construction, by the process of \emph{para\-bolic induction}, of further representations, sufficiently   many in number  to  decompose the regular representation.  
\end{enumerate}
The cuspidal representations of a real reductive group $G$ are the irreducible and  unitary representations of $G$ that are  square-integrable, modulo center.  Their classification fits well with ideas from $C^*$-algebra $K$-theory and noncommutative geometry.  Indeed the classification was an important source of inspiration for the formulation of the Baum-Connes conjecture.  But the functor  of parabolic induction has   received less attention from operator algebras and noncommutative geometry.  Our purpose here is to continue the effort begun in \cite{Clare} and in \cite{CCH1,CCH2} to address this imbalance, if only modestly.

A few years ago Pierre Clare explained in \cite{Clare} how parabolic induction fits into the theory of Hilbert $C^*$-modules and bimodules in a way that is very similar to  Marc Rieffel's well known treatment of ordinary induction \cite{Rieffel}.  In joint work with Clare \cite{CCH1,CCH2} we studied parabolic induction as a functor between categories of Hilbert $C^*$-modules.    Using a considerable amount of representation theory, due to Harish-Chandra, Langlands and others, we constructed an adjoint Hilbert $C^*$-bimodule,  used it to define a functor of \emph{parabolic restriction} between categories of Hilbert $C^*$-modules, and proved that parabolic induction and restriction are two-sided ``local'' adjoints of one another.

 A drawback of our work was that the concept of local adjunction is significantly weaker than the standard category-theoretic notion of adjunction.      This shortcoming was unavoidable: there is no category-theoretic adjunction at the Hilbert $C^*$-module level. Moreover  the natural candidates for the unit maps of the sought-for adjunctions are even not properly defined at the Hilbert $C^*$-module level.
 
 The purpose of this article is to examine the extent to which the shortcomings of the Hilbert $C^*$-module theory can be remedied by adjusting the context a little.    To this end we shall study new categories consisting of group or $C^*$-algebra representations on operator spaces.  We shall prove that the new categories have only the familiar  irreducible objects, so that they present a plausible context for representation theory.  Then we shall formulate and prove a simple theorem about adjoint pairs of functors between the new categories of operator space modules over $C^*$-algebras (as opposed to Hilbert $C^*$-modules).   As we shall explain, this implies  in a very simple way (that does not require any sophisticated representation theory) a Frobenius reciprocity theorem for parabolic induction (the theorem is that the functor of parabolic induction has a left adjoint, which we shall describe  explicitly, along with the adjunction isomorphism).
 
 Secondly, we shall examine in detail the form of parabolic induction and restriction at the level of   Harish-Chandra's Schwartz algebra   in the particular case where $G=SL(2,\R)$.  We shall summarize the tempered representation theory  of $G$ in the form of a Morita equivalence between the Harish-Chandra algebra and a simpler and more accessible algebra.  Our reason for doing this is to  formulate and prove    a ``second adjoint theorem'' for tempered representations in this case, along the lines of Bernstein's fundamental second adjoint theorem (that parabolic induction also has an explicit right adjoint) in the smooth representation theory of reductive $p$-adic groups \cite{Bernstein-p-adic,Bernstein-2-adjoint}.
 
 We shall say a good deal more elsewhere about our second adjoint theorem for tempered representations.  Our reason for introducing the result  here is to use it as a test for measuring the potential usefulness of new operator-algebraic contexts for representation theory. { To this end, we shall conclude by using the explicit formulas obtained for the Harish-Chandra algebra to explore the prospects for an elaboration of the operator space module Frobenius   reciprocity relation analyzed in Section~\ref{sec-op-sp-mod}  so as to include Bernstein's second adjunction.  We shall give one concrete suggestion about how this might be achieved in Section~\ref{sec-op-alg-bernstein}.}

\section{Categories of  Operator Space Modules}
\label{sec-op-sp-mod}

We shall study operator space modules over $C^*$-algebras and, later on, over operator algebras. For the most part we shall refer to the  monograph \cite{BlM} for background on operator spaces, but we shall repeat some of the  basic definitions here.

But before we start, let us explain our point of view.  In the  representation theory of real reductive groups  there is broad agreement about the concepts of \emph{irreducible} representation that are appropriate for study, along with the associated concepts of equivalence among irreducible representations.  But representations that lie well beyond the irreducible representations are little-studied in representation theory.  From the point of view of noncommutative geometry this is an awkward omission, since for example the $K$-theory studied in noncommutative geometry, and used to formulate the Baum-Connes conjecture, involves representations that are far from irreducible.   So it is of interest to explore some of the potentially convenient categories of representations that operator algebra theory provides. 

As we mentioned in the introduction, our  immediate concern here is not $K$-theory but  {parabolic induction, together with adjunction theorems such as Frobenius reciprocity.  But here, too, the choice of a category of representations matters.  Our main observation is that operator spaces  can offer a very convenient starting point from    which to begin an examination of Frobenius reciprocity and related matters, because the theorem assumes a particularly elementary form there.

To continue, recall that an  \emph{operator space} is a complex vector space $X$ equipped with a family of Banach space norms on the spaces $M_n(X)$ of $n\times n$ matrices over $X$ that satisfy the following two conditions:
\begin{enumerate}[\rm (a)]
 \item If   $x\in M_n(X)$ and $a,b\in M_n(\mathbb C)$, 
 then 
\[
\| axb \| \leq \| a\| \|x\| \|b\|.
\]
 \item The norm of a block-diagonal matrix is the maximum of the norms of the diagonal blocks.
\end{enumerate}
A linear map $T:X\to Y$ between operator spaces induces maps 
\[
M_n(T):M_n(X)\longrightarrow M_n(Y)
\]
 by applying $T$ to each matrix entry, and we say that $T$ is \emph{completely bounded} (c.b.) if 
 \[
 \sup_{n} \| M_n(T)\|_{\text{operator}} <\infty.\]
The supremum is the \emph{completely bounded norm}. 
We shall also use the related notions  of \emph{completely contractive} and \emph{completely isometric} map.

\begin{example}
\label{ex-col-hilb-sp}
 Every Hilbert space $H$ carries a number of operator space structures. In this paper  we shall consider only the \emph{column Hilbert space} structure, in which $H$ is identified with the concrete operator space $\Bounded(\mathbb C, H)$ of bounded operators from $\mathbb C$ to $H$. Every bounded operator between Hilbert spaces is completely bounded as an operator between column Hilbert spaces, and the completely  bounded norm is the operator norm.  See  \cite[1.2.23]{BlM}. \end{example}

Let $X$, $Y$ and $Z$ be operator spaces.  A bilinear map $\Phi \colon X\times Y\to Z$ gives  rise to bilinear maps 
\[
M_n (\Phi)\colon M_n(X) \times  M_n (Y) \longrightarrow M_n (Z)
\]
through the formula 
\[
M_n (\Phi)\colon \bigl  (  [x_{ij}], [y_{ij}]\bigr  )  \longmapsto \bigl  [\textstyle{\sum_{k=1}^{n} \Phi (x_{ik},y_{kj})} \bigr ]. 
\]
The \emph{Haagerup tensor product} $X\otimes_h Y$  is a completion of the algebraic tensor product over $\mathbb C$ and an operator space,  characterized by the property that every completely contractive $\Phi$ as above   factors uniquely through a completely contractive map $X\otimes_h Y\to Z$.
 See \cite[1.5.4]{BlM}.

\begin{definition}
An \emph{operator algebra} is an operator space $A$ which is also a Banach algebra with a bounded approximate unit, such that the product in $A$ induces a completely contractive map $A\otimes_h A\to A$.
\end{definition}

\begin{definition}
 A \emph{\textup{(}left\textup{)} operator module} over an operator algebra $A$ is an operator space $X$ and a nondegenerate\footnote{\emph{Nondegenerate} means that $A\cdot X$ is dense in $X$.}  left $A$-module for which the module action extends to a completely contractive map $A\otimes_h X\to X$.  One similarly defines right operator modules, and operator bimodules.   
We shall denote by   $\OMod_A$  the category of left operator modules over $A$ and  completely bounded $A$-module maps.
\end{definition}

\subsection{Irreducible Operator Modules}

If $A$ is a $C^*$-algebra, then the category   $\OMod_A$  contains the category $\HMod_A$ of  nondegenerate  Hilbert space representations of $A$ as a full subcategory (each Hilbert space being given its column operator space structure). 
Our first observation  is that if $A$ is of type I, then $\OMod_A$ and $\HMod_A$ have the same irreducible objects (that is, the same modules having no nontrivial closed submodules):

\begin{proposition}\label{same_irreducibles_proposition}
Let $A$ be a type I $C^*$-algebra. Every   irreducible operator $A$-module  is completely isometrically isomorphic in $\OMod_A$ to an irreducible  Hilbert space representation of $A$. 
\end{proposition}

\begin{proof}
Let $X$ be an  irreducible operator $A$-module.  By \cite[Theorem 3.3.1]{BlM} there is a nondegenerate representation of $A$ on a Hilbert space $H$,  a second Hilbert space $K$, and a completely isometric isomorphism from $X$ to a closed $A$-submodule of the space $\Bounded(K,H)$ of bounded operators from $K$ to $H$.  We shall realize $X$ as a subspace of  $\Bounded(K,H)$ in this way, and we may assume that $X\cdot K$ is dense in $H$.  

We are going to argue that the representation of $A$ on $H$ is a multiple of a single irreducible representation of $A$.  To begin, the representation of $A$ on $H$ extends to the multiplier algebra $M(A)$, and the restriction to the center $Z(M(A))$  is a multiple of a single irreducible representation of the center.  For otherwise there would exist elements $z_1,z_2\in Z(M(A))$ with $z_1z_2=0$ yet 
\[
z_1H\neq 0   \quad \text{and} \quad z_2H\neq 0,
\] 
so that 
\[
z_1X\neq 0   \quad \text{and} \quad z_2X\neq 0 .
\] 
The subspace 
\[
  \overline{z_1X}   \subseteq X
\]
would then be a nontrivial submodule of the supposedly irreducible module $X$.

Assume now that $A$ is liminal (that is, $A$ acts as compact operators in each irreducible representation) and that furthermore the spectrum $\widehat A$ is a Hausdorff topological space.
The  Dauns-Hofmann theorem (see \cite{DaunsHofmann} or \cite[Section 4.4]{Pedersen}) identifies $Z(M(A))$ with the algebra of bounded, continuous, com\-plex-valued functions on  $\widehat A$. The identification is as follows:
\begin{enumerate}[\rm (a)]
\item  If $I\subseteq Z(M(A))$ is the maximal ideal corresponding to evaluation at $[\pi]\in \widehat A$, then $\overline{IA}$ is the kernel of $\pi$. 
\item  In contrast, if $I\subseteq Z(M(A))$ is the maximal ideal corresponding to evaluation at a point at infinity, then $\overline{I A} = A$.   
\end{enumerate}
In our present situation, we see that the action of $Z(M(A))$ on $H$ must factor through the quotient by a maximal ideal corresponding to a point of $\widehat A$, as in item (a), since the action of $A$ on $H$ is certainly nonzero. Therefore   the action of $A$ must factor through a quotient  $A/ \overline{IA}$. Since the quotient is isomorphic to the compact operators, the action of $A$ is a multiple of a single irreducible representation, as required.  

Next assume that $A$ is a general liminal $C^*$-algebra. Let $\{J_\alpha\}$ be a composition series for $A$  for which each quotient $J_{\alpha+1}/J_\alpha$ has Hausdorff spectrum. Take the least $\alpha$ for which $J_\alpha X\neq 0$; by irreducibility we must then have $J_\alpha X=X$. Consider $X$ as an irreducible operator module over $B=J_\alpha/J_{\alpha-1}$, where as usual $J_{\alpha-1}$ denotes the closure of the union of all ideals in the composition series smaller than $J_\alpha$.  The argument above shows that $H$ is a multiple of a single irreducible representation of $B$, and hence is a multiple of a single irreducible representation of $A$.

Finally, if  $A$ is a  general type I $C^*$-algebra,  we can  apply the above argument to a composition series for $A$ with liminal quotients to show that $H$ is a multiple of a single irreducible representation of $A$ in this case too.

We have now shown in general that  $H$ is a multiple of a single irreducible representation of $A$, say
\[
H \cong  M\otimes L   
\]
where $A$ acts trivially on $M$ and irreducibly on $L$. We shall now show that there is a bounded operator $S:K\to M$ such that $X$ consists precisely of all operators in $\Bounded(K,M\otimes L)$ of the form
\begin{equation}\label{S_equation}  k\longmapsto S(k)\otimes \ell,\end{equation}
as $\ell$ ranges over $L$. The map sending the operator \eqref{S_equation} to $\|S\|\cdot\ell$ will then be a completely isometric $A$-linear isomorphism from $X$ to $L$.

For each $k\in K$ and $m\in M$ consider the completely bounded, $A$-linear map from $X$ to $L$ defined by
\begin{equation}\label{eq-iso-phi}
X\ni T \longmapsto m^* \cdot T \cdot k \in L,
\end{equation}
where $m^*:M\otimes L\to L$ is the operator $m'\otimes \ell \mapsto \langle m,m'\rangle \ell$. Fix $k_0$ and $m_0$ for which the operator \eqref{eq-iso-phi} is nonzero. Then this operator is invertible: its kernel is a proper closed submodule of the irreducible module $X$, while its image is a nonzero $A$-invariant subspace of $L$, which must equal $L$ by Kadison's transitivity theorem \cite{Kadison} (or by a direct argument in the present rather elementary type I situation).

Applying Schur's lemma to the irreducible representation $L$, we find that there are scalars $c_{k,m}$ such that
\[
 m^* \cdot T \cdot k=c_{k,m} \cdot m_0^* \cdot T \cdot k_0
\]
for all $k\in K$ and $m\in M$. Taking $S\in \Bounded(K,M)$ to be the operator defined by $\langle m,S(k)\rangle = c_{k,m}$, we have
\[ T(k) = S(k)\otimes (m_0^*\cdot T\cdot k_0)\]
for all $T\in X$ and all $k\in K$.
\end{proof}

\begin{remark} We do not know if the assumption in the previous proposition that $A$ be of  type I is actually necessary.
\end{remark}

\subsection{Functors Between Operator Module Categories}

If $X$ is a right operator  $B$-module, and if $Y$ is a left  operator $B$-module, then we can of course form the algebraic tensor product of $X$ and $Y$ over $B$.  The \emph{balanced Haagerup tensor product} $X\otimes_{hB}Y$ is a completion of the algebraic tensor product and an operator space, characterized  by the fact that any completely contractive bilinear map 
\[
 \Phi \colon X\times Y\longrightarrow  Z 
\]
with $\Phi (xb,y) = \Phi (x,by)$ extends to a completely contractive map from $X\otimes_{hB} Y$ to $Z$. The Haagerup tensor product is associative, and functorial with respect to c.b.\ bimodule maps. If $X$ and  $Y$ carry left operator $A$-module and right operator $C$-module structures, respectively, then $X\otimes_{hB} Y$ is an operator $A$-$C$-bimodule. See \cite[Section 3.4]{BlM}.

Now let $A$ and $B$ be operator algebras, and let $E$ be an operator $A$-$B$-bimodule.  We obtain a functor 
\[
\OMod_B \longrightarrow \OMod_A
\]
from the Haagerup tensor product operation:
\[
X \longmapsto E\otimes _{hB} X.
\]
When needed, we shall give the functor the same name---$E$---as the bimodule. 
Note that composition of functors corresponds, up to natural isomorphism, to Haagerup tensor product of bimodules.

If $X$ is an operator $A$-module, then the module structure    induces a completely isometric isomorphism 
\begin{equation}
\label{eq-om-smooth1}
A\otimes_{hA} X\stackrel \cong \longrightarrow X  .
\end{equation}
Similarly, we obtain a completely isometric isomorphism  
\begin{equation}
\label{eq-om-smooth2}
 X\otimes_{hB} B \stackrel \cong \longrightarrow X 
\end{equation}
in the case of a right operator $B$-module structure.  See \cite[Lemma 3.4.6]{BlM}. So tensoring with $A$ or $B$, viewed as   operator bimodules over themselves, gives the identity functor up to natural isomorphism.

\subsection{Adjunctions} 
Our aim is to study adjunction relations, in  the usual sense of category theory, between the functors introduced above. 
So let $A$ and $B$ be operator algebras,  and let $E$ be an operator $A$-$B$-bimodule.  In addition,  let $F$ be an operator $B$-$A$-bimodule.  Following standard terminology, we say that $F$ is \emph{left adjoint} to $E$, and that $E$ is \emph{right adjoint} to $F$, if  there is a natural isomorphism 
\begin{equation}
\label{eq-adjunction-def}
\CB_B(F\otimes_{hA} X, Y) \stackrel \cong \longrightarrow \CB_A(X, E\otimes_{hB} Y),
\end{equation}
as $X$ ranges over  all operator $A$-modules and $Y$ ranges over  all operator $B$-modules.  

The bijection is required to be simply a bijection of sets, but in fact it is automatically a uniformly (over $X$ and $Y$) completely bounded  natural isomorphism of operator spaces, as the following simple lemmas make  clear (the lemmas simply place the unit/counit characterization of adjoint functors within the operator module context: compare \cite[Chapter IV]{MacLane}).

\begin{lemma}  Associated to each natural isomorphism \eqref{eq-adjunction-def} 
there is a   completely bounded $A$-bimodule map 
\[
\eta \colon A \longrightarrow  E\otimes_{hB} F
\]
 \textup{(}the \emph{unit} of the adjunction\textup{)}  with the property that the composition 
 \[
 \CB_B(F\otimes_{hA} X, Y) \stackrel{ \text{\eqref{eq-adjunction-def}} } \longrightarrow \CB_A(X, E\otimes_{hB} Y) 
\stackrel \cong\longrightarrow  \CB_A(A \otimes_{hA}X , E\otimes_{hB} Y)
 \]
sends a morphism $T\colon F\otimes_{hA} X\to  Y$ to the composition
 \[
 A\otimes _{hA} X \stackrel{\eta\otimes \mathrm {id}}\longrightarrow E\otimes_{hB} F \otimes _{hA} X  \stackrel{\mathrm{id}\otimes T } \longrightarrow E \otimes _{hB} Y .
 \]
 \end{lemma}

 \begin{proof}
 The definition of $\eta$ is very simple (and standard).   Take $X = A$ and $Y=F$ in the isomorphism \eqref{eq-adjunction-def}  to obtain 
 \[
 \CB_B(F\otimes_{hA} A, F) \stackrel \cong \longrightarrow \CB_A(A, E\otimes_{hB} F).
 \]
 Then define $\eta$ to be the image on the right hand side of the canonical element $F\otimes_{hA} A\to F$ on the left given by the module action.  The map $\eta$ defined in this way is \emph{a priori} just a left $A$-module map, but the naturality of \eqref{eq-adjunction-def} implies it is a right $A$-module map too.
 
 The proof that 
 the isomorphism \eqref{eq-adjunction-def} is given by the formula in the lemma is a straightforward consequence of naturality of the isomorphism once again, together with the following claim: the isomorphism 
  \begin{multline*}
 \CB_B(F\otimes_{hA} X, F\otimes _{hA}X) \stackrel {\text{\eqref{eq-adjunction-def}}} \longrightarrow \CB_A(X, E\otimes_{hB} F\otimes_{hA} X) \\
 \stackrel \cong \longrightarrow \CB_A(A\otimes _{hA} X, E\otimes_{hB} F\otimes_{hA} X)
 \end{multline*}
takes the identity operator on $F \otimes_{hA}X$ to $\eta\otimes \mathrm{id}_X$.  As for the claim, denote by  
\[
S\colon X \longrightarrow  E\otimes_{hB} F\otimes_{hA} X
\]
the image of the identity operator on $F \otimes_{hA}X$  under \eqref{eq-adjunction-def}.  From the commuting  diagram
\[
\xymatrix{
 \CB_B(F\otimes_{hA} X, F\otimes _{hA}X)\ar[d] \ar[r] 
 &  \CB_A(X, E\otimes_{hB} F\otimes_{hA} X) \ar[d] \\
  \CB_B(F\otimes_{hA} A, F\otimes _{hA}X) \ar[r] &  \CB_A(A, E\otimes_{hB} F\otimes_{hA} X)\\
    \CB_B(F\otimes_{hA} A, F\otimes _{hA}A) \ar[u] \ar[r] &
      \CB_A(A, E\otimes_{hB} F\otimes_{hA} A)\ar[u]
}
\]
 in which both squares are associated, by the naturality of \eqref{eq-adjunction-def}, to  a c.b.\  $A$-module map from $A$ into $X$, we see that $S$ is equal to $\eta\otimes \mathrm{id}_X$ on the image of any $A\to X$.  But these images are dense in $X$, so the claim is proved.
 \end{proof}

Similarly:

\begin{lemma}  Associated to each natural isomorphism \eqref{eq-adjunction-def} 
there is   a completely bounded $B$-bimodule map 
 \[
 \epsilon \colon F\otimes_{hA} E \longrightarrow  B
 \]
 \textup{(}the \emph{counit} of the adjunction\textup{)}  with the property that the inverse of the composition 
  \[
 \CB_A(X, E\otimes_{hB} Y)
  \stackrel{ \text{\eqref{eq-adjunction-def}} } \longleftarrow
  \CB_B(F\otimes_{hA} X, Y)
\stackrel \cong\longleftarrow  \CB_B(F \otimes_{hA}X , B\otimes_{hB} Y)
 \]
sends a morphism $S\colon   X\to  E\otimes_{hB} Y$ to the composition
 \[
F \otimes _{hA}  X\stackrel{\mathrm{id}\otimes S }\longrightarrow F\otimes_{hA} E \otimes _{hA} Y  \stackrel{\varepsilon \otimes \mathrm{id} } \longrightarrow     B\otimes _{hB} Y .
 \eqqed
 \]
 \end{lemma}

The unit and counit of an adjunction are linked by standard identities, and conversely any appropriate pair of linked bimodule maps gives an adjunction: 

\begin{lemma}  
\label{lem-linking}
Given an adjunction \eqref{eq-adjunction-def} and maps 
\[
 \eta \colon A \longrightarrow  E\otimes_{hB} F \qquad \text{and}\qquad 
 \epsilon \colon F\otimes_{hA} E \longrightarrow  B
 \] 
 as in the previous lemmas, the two compositions
 \[ 
 A\otimes_{hA} E \xrightarrow{\eta\otimes \mathrm{id}} E\otimes_{hB} F\otimes_{hA} E \xrightarrow{\mathrm{id}\otimes\epsilon} E\otimes_{hB} B   \longrightarrow E 
 \]
 and
 \[ 
 F\otimes_{hA} A \xrightarrow{\mathrm{id}\otimes \eta} F\otimes_{hA} E\otimes_{hB} F \xrightarrow{\epsilon\otimes \mathrm{id}} B\otimes_{hB} F  \longrightarrow F 
 \]
 are the the canonical isomorphisms induced from the left and right $A$- and $B$-module actions on $E$ and $F$, respectively.   Conversely this data determines an adjunction isomorphism. \qed
 \end{lemma}
 
  For the proof, compare for example \cite[Chapter IV]{MacLane} once again.

\begin{example}\label{operator_Frobenius_example}
 Let $A$ be a closed subalgebra of $B$ satisfying 
 \[
\overline{AB}= B=\overline{BA}.
 \]
  Let $E=B$, considered as an operator $A$-$B$-bimodule. The corresponding tensor product functor 
  \[
 \OMod_B \longrightarrow \OMod_A 
  \]
simply associates to an operator $B$-module   its restriction to an operator $A$-module. Then define  $F=B$, considered as an operator $B$-$A$-bimodule.  The associated tensor product  functor $X\mapsto B\otimes_A X$ is left adjoint to $E$. The maps 
\[
 \eta \colon A\longrightarrow  E\otimes_{hB} F  \qquad \text{and}\qquad 
 \epsilon \colon F\otimes_{hA} E\longrightarrow  B
 \]
 given by the formulas $\eta(a_1a_2)=a_1\otimes a_2$ and $\epsilon(b_1\otimes b_2) = b_1 b_2$ are the unit and counit of an adjunction. 
\end{example}

\subsection{An Adjunction Theorem from Hilbert C*-Modules}

Hilbert $C^*$-modules provide a very simple set of instances of the ideas from the previous section.   To see this, we need to first recall some elegant observations, due to Blecher \cite{Blecher-cstar}, that link operator spaces to Hilbert $C^*$-modules.  See also \cite[Chapter 8]{BlM}, as well as \cite{Lance}   for an introduction to Hilbert $C^*$-modules.

Let $E$ be a right Hilbert $C^*$-module over a $C^*$-algebra $B$.  The matrix space $M_n(E)$ is naturally a Hilbert $C^*$-module over $M_n(B)$, with inner product 
\[
\bigl \langle [e_{ij}], [f_{ij}]\bigr \rangle = \bigl [\,{\textstyle \sum _{k} \langle e_{ki},f_{kj}\rangle }\,\bigr ] ,
\]
and in this way we give $E$ the structure of an operator space and a right operator $B$-module.

 A bounded, adjointable operator   between Hilbert $C^*$-$B$-modules is automatically completely bounded with the same norm (in fact this is true for any bounded $B$-module map, whether  or not it is adjointable).  
 
 We are especially interested in
the situation where a Hilbert $C^*$-$B$-module $E$ is equipped with a left action of a second $C^*$-algebra $A$ by bounded and adjointable operators. One sometimes  calls $E$ a \emph{$C^*$-correspondence} from $A$ to $B$, and every such correspondence is an operator $A$-$B$-bimodule.

Now if $E$ is \emph{any} operator space, then its \emph{adjoint} $E^*$ is the complex conjugate vector space, equipped with the norms 
\[
\bigl \| [e_{ij} ] \bigr \|_{M_n (E^*)} = \bigl \| [ e_{ji} ]\bigr \|_{M_n(E)},
\]
which endow $E^*$ with the structure of an operator space.  See \cite[Section 1.2.25]{BlM}. If $E$ is an operator $A$-$B$-bimodule, where $A$ and $B$ are $C^*$-algebras, then $E^*$ is an operator $B$-$A$-bimodule via the formula
\[ 
 b\cdot   e^*\cdot a  =  ( a^* \cdot e\cdot  b^*)^*.
\]
 
Let us  apply this construction to the situation in which $E$ is a Hilbert $C^*$-$B$-module, as follows.  Denote by $\Compact_B(E)$ the $C^*$-algebra of $B$-compact operators on $E$, that is, the closed linear span of all bounded adjointable operators on $E$ of the form
\[
e_1^{\vphantom{*}}\otimes e_2^* \colon e \longmapsto  e_1 \langle e_2 , e \rangle .
\]
The tensor product notation is particularly apt in view of the following very elegant and   useful calculation of Blecher. 
\begin{lemma}
\label{lem-cb-compact}
 \cite[Corollary 8.2.15]{BlM}.
The above formula defines a completely contractive map
\begin{equation*}
\kappa \colon E  \otimes_{hB}  E^*\longrightarrow   \Compact_B(E) ,
\end{equation*}
and this map  is in fact a completely isometric isomorphism. \qed
\end{lemma}

The Haagerup tensor product also fits with Hilbert module theory in a second way:\footnote{A third very elegant connection, which like the first is due to Blecher, will be indicated in Lemma~\ref{lem-Blecher-again}.}
\begin{lemma}
\label{lem-cb-ip}
 \cite[Lemma 3.16]{CCH1}.
If $E$ is a $C^*$-$A$-$B$-corresp\-ondence, then  the inner product induces a   completely contractive map 
\[
E^*\otimes_{hA}  E \longrightarrow B  
\]
of operator $B$-$B$-bimodules. \qed
\end{lemma}

The lemmas lead to the following simple, sufficient condition for a $C^*$-correspondence $E$ to admit a left adjoint when viewed as a functor
\[
E\colon \OMod_B\longrightarrow \OMod_A.
\]

\begin{theorem}\label{operator_Frobenius_theorem}
Let $A$ and $B$ be $C^*$-algebras, and let $E$ be a $C^*$-correspondence from $A$ to $B$.  If the action of $A$ on $E$ is through  $B$-compact operators, then the operator $B$-$A$-bimodule $E^*$ is left adjoint to $E$.
\end{theorem}

\begin{proof}
The action of $A$ on $E$ gives rise to a $*$-homomorphism
\[
\alpha \colon A \longrightarrow \Compact_B(E),
\]
and hence, by Lemma~\ref{lem-cb-compact}, to a c.b. $A$-$A$-bimodule map
\[
\eta \colon A \longrightarrow E\otimes_{hB} E^* .
\]
On the other hand  by Lemma~\ref{lem-cb-ip} the inner product on $E$ gives us  a c.b.\ $B$-bimodule map
 \[ 
 \epsilon \colon  E^*\otimes_{hA} E \longrightarrow  B .
 \]
 We claim that these are the unit and counit, respectively of an adjunction.  
According to Lemma~\ref{lem-linking} to prove this it suffices to show that the compositions
 \begin{equation}
 \label{eq-linking1}
 A\otimes_{hA} E \xrightarrow{\eta\otimes \mathrm{id} } E\otimes_{hB} E^* \otimes_{hA} E \xrightarrow{\mathrm{id}\otimes\epsilon} E\otimes_{hB} B   \longrightarrow E 
 \end{equation}
 and
 \begin{equation}
 \label{eq-linking2}
E^* \otimes_{hA} A \xrightarrow{\mathrm{id}\otimes \eta} E^*\otimes_{hA} E\otimes_{hB} E^*  \xrightarrow{\epsilon\otimes \mathrm{id}} B\otimes_{hB} E^* \longrightarrow E^*
 \end{equation}
 are the the canonical isomorphisms induced from the left and right $A$- and $B$-module actions on $E$ and $E^*$, respectively.

The composition 
\[
E\otimes_{hB} E^* \otimes_{hA} E \stackrel{\mathrm{id}\otimes \epsilon}\longrightarrow  E\otimes_{hB} B \stackrel{\cong} \longrightarrow E
\]
 is given on elementary tensors by the formula 
 \[
 e_1^{\vphantom{*}}\otimes e_2^*  \otimes e_3^{\vphantom{*}}
  \mapsto e_1\langle e_2,e_3\rangle = \kappa(e_1^{\vphantom{*}}\otimes e_2^*)e_3^{\vphantom{*}},
 \]
 where $\kappa$ is the completely isometric isomorphism of Lemma~\ref{lem-cb-compact}.
On the other hand, the map 
  \[
  A\otimes_{hA} E \stackrel{\eta\otimes \mathrm{id}}\longrightarrow  E\otimes_{hB} E^* \otimes_{hA} E
  \]
is given by the formula 
 \[
 a\otimes e\in E \longmapsto   \kappa^{-1}(\alpha(a)) \otimes e.
 \]
  Combining these two computations, we find that  the   composition \eqref{eq-linking1} is
  \[
 A\otimes _{hA} E \ni   a\otimes e\longmapsto   \kappa^{-1}(\alpha(a)) \otimes e \longmapsto 
    \alpha(a)e \in E,
  \]
  as required. The second composition  \eqref{eq-linking2} is treated similarly.  The composition 
  \[
   E^*\otimes_{hA} E\otimes_{hB} E^*  \xrightarrow{\epsilon\otimes \mathrm{id}} B\otimes_{hB} E^* \longrightarrow E^*
   \]
   is given by the formula
 \[
 e_1^{*}\otimes e_2^{\vphantom{*}}  \otimes e_3^{*}
  	 \mapsto \langle e_1^{\vphantom{*}}, e_2^{\vphantom{*}}\rangle e_3^* 
  	 =\left (  e_3^{\vphantom{*}} \langle e_2^{\vphantom{*}}, e_1^{\vphantom{*}}\rangle\right ) ^* 
   	 = \left ( \kappa(e_2^{\vphantom{*}}\otimes e_3^*)^*e_1\right ) ^{*},
 \]
   while the map 
   \[
   E^* \otimes_{hA} A \xrightarrow{\mathrm{id}\otimes \eta} E^*\otimes_{hA} E\otimes_{hB} E^*
   \]
 is given by the formula 
 \[
 e^*\otimes a \longmapsto e^* \otimes \kappa^{-1}(\alpha (a)).
 \]
  So the composition \eqref{eq-linking2} is 
  \[
  e^*\otimes a \longmapsto e^* \otimes \kappa^{-1}(\alpha (a))
  \longmapsto ( \alpha(a)^* e)^* ,
  \]
  and the image is $e^*\alpha (a)$, as required.
\end{proof}

\section{Operator Modules and Parabolic Induction}
\label{sec-bimodules}
 
We turn now to representations of groups.  Let $G$ be a real reductive group. For definiteness, let us assume, more precisely, that $G$ is the group of real points of a connected reductive group defined over 
$\R$, as we did in \cite{CCH1}, although what we have to say would certainly apply to a broader class of examples.   On the other hand the special linear and general linear groups will suffice to illustrate the results of this paper. 

We shall be  interested in (continuous) unitary representations of $G$, and usually, in particular, in representations that are weakly contained in the regular representation, and so correspond to nondegenerate representations of the reduced $C^*$-algebra of $G$.

Let $P$ be a parabolic subgroup of $G$, with Levi decomposition 
\[
P=LN.
\]
For example if $G$ is a general linear, or special linear, group,  then up to conjugacy $P$ is a subgroup of block-upper-triangular matrices, $L$ is the subgroup of block-diagonal matrices, and $N$ is the subgroup of block-upper-triangular matrices with identity diagonal blocks. See \cite[Section VII.7]{KnappBeyond} for the general definitions.

The functor of \emph{\textup{(}normalised\textup{)} parabolic induction},
\[
\Ind_P^G \colon \HMod_{C^*(L)} \longrightarrow \HMod_{C^*(G)} ,
\]
associates to a unitary representation $\pi\colon L\to U(H)$ (or equivalently, a nondegenerate representation of the full group $C^*$-algebra) the Hilbert space completion of the space of continuous functions
\[
\bigl \{ \, f \colon G \to H \, : \, f (g\ell n ) = \pi(\ell)^{-1} \delta(\ell)^{-\frac 12} f (g) \, \bigr \} ,
\]
where 
\[
\delta(\ell) = \det \bigl(\operatorname{Ad}_\ell\colon \mathfrak n\to \mathfrak n\bigr ) ,
\]
 in the inner product 
\[
\langle f_1 ,f_2\rangle = \int _K \langle f_1(k) ,f_2(k)\rangle _H \, dk ,
\]
where $K$ is a maximal compact subgroup of $G$.   The presence of the normalizing factor $\delta ^{-\frac 12}$ ensures that the Hilbert space so obtained is a unitary representation of $G$ under the left translation action.  If the original representation is weakly contained in the regular representation, then so is the parabolically induced representation.  For all this see for example \cite[Chapter VII]{knapp-overview}.

\subsection{Parabolic Induction and Hilbert C*-Modules}
Pierre Clare began the study of parabolic induction from the point of view of modules and bimodules over operator algebras in   \cite{Clare}.

 Clare realized the functor of normalised parabolic induction as the  tensor product with an explicit $C^*$-correspondence $C^*_r(G/N)$,  from $C^*_r (G)$ to $C^*_r(L)$, which is obtained as a completion of the space of continuous, compactly supported functions on the homogeneous space $G/N$ in a natural (normalized, using $\delta$) inner product valued in $C^*_r(L)$.  Thus he exhibited a natural isomorphism 
 \[
 \Ind _P^G H \cong  C^*_r (G/N) \otimes _{C^*_r (L)} H  
 \]
 of functors from $\HMod_{C^*_r (L)}$ to $\HMod _{C^*_r (G)}$.  See \cite[Section 3]{Clare} or \cite[Section 4]{CCH1}.

 \begin{remark} 
 Actually Clare considered the \emph{full} group $C^*$-algebra in \cite{Clare}.  Here we shall follow the approach in \cite{CCH1} and work with the \emph{reduced} $C^*$-algebra, and the associated reduced version of Clare's bimodule.  The theorem that we shall present below holds in either context, but for later purposes it is more appropriate for us to work with the reduced $C^*$-algebra.
 \end{remark} 
 
 The Hilbert module   picture of parabolic induction as a tensor product allows us to define parabolic induction of operator modules, 
  \[
  \Ind _P^G  \colon \OMod_{C^*_r(L)} \longrightarrow \OMod_{C^*_r (G)}  
 \]
 using the Haagerup tensor product:
 \[
   \Ind _P^G X = C^*_r (G/N) \otimes_{hC^*_r (L)} X.
   \]

\begin{remark} 
By a famous theorem of Harish-Chandra \cite[Theorem 6, p.230]{HC-liminal}, every real reductive group  is of type I; indeed it is liminal. So Proposition~\ref{same_irreducibles_proposition}, concerning irreducible objects in the categories $\OMod_{C^*(G)}$ and  $\OMod_{C^*_r(G)}$ applies. 
\end{remark}

Within the context of operator modules it is natural and simple to consider in addition to $C^*_r (G/N)$ the adjoint operator space $C^*_r(G/N)^*$, which is  an operator $C^*_r (L)$-$C^*_r (G)$-bimodule .  We obtain from the tensor product formula 
\[
   \Res_P^G X = C^*_r (G/N)^* \otimes_{hC^*_r (G)} X 
   \]
a functor
\[ \Res _P^G  \colon \OMod_{C^*_r(G)} \longrightarrow \OMod_{C^*_r (L)}  ,
\]
that we shall call \emph{parabolic restriction}.

\begin{theorem}\label{G_operator_Frobenius_theorem}
Parabolic restriction is left-adjoint to parabolic induction, as functors on operator modules.  Thus there is a natural isomorphism 
\[ \CB_{C^*_r(L)}(   \Res_P^G X, Y) \cong \CB_{C^*_r(G)}(X,  \Ind _P^G  Y)
\]
for all operator $C^*_r(G)$-modules $X$ and all operator $C^*_r(L)$-modules $Y$. 
\end{theorem}

\begin{proof}
 In \cite[Proposition 4.4]{CCH1} we showed that the action of $C^*_r(G)$ on the $C^*$-correspondence $C^*_r(G/N)$ is by  through compact operators. The result is therefore an immediate consequence of Theorem \ref{operator_Frobenius_theorem}.
\end{proof}

\begin{remark}
The same argument shows  that $C^*(G)$ acts by compact operators on $C^*(G/N)$, and so there is an analogue of Theorem \ref{G_operator_Frobenius_theorem} for operator modules over the full group $C^*$-algebras.
\end{remark}

\subsection{Local Adjunction}
\label{sec-local-adj}
Let us contrast the theorem proved in the previous section with the situation for categories of Hilbert $C^*$-modules.

In \cite{CCH1} we were able to show, using considerable input from representation theory, that the operator bimodule $C^*_r(G/N)^*$ in fact carries the structure of a $C^*$-correspondence.  In other words  its operator space structure is induced from a $C^*_r(G)$-valued inner product. 

It needs to be stressed that this circumstance depends in a delicate way on issues in representation theory; in fact our explicit formula for the inner product is derived from  Harish-Chandra's Plancherel formula.  There is for example no similar inner product within the context of \emph{full} group $C^*$-algebras.

In any case, we can use Kasparov's interior tensor product operation \cite[Chapter 4]{Lance} to define parabolic induction and restriction functors
\[  \Ind\colon \CMod_{C^*_r(L)} \longrightarrow  \CMod_{C^*_r(G)} \]
and 
\[ \Res\colon \CMod_{C^*_r(G)} \longrightarrow \CMod_{C^*_r(L)}\]
between categories of (right) Hilbert $C^*$-modules and adjointable operators between Hilbert $C^*$-modules.  

Kasparov's interior tensor product is related to the Haagerup tensor product in a very simple way:
\begin{lemma}\cite[Theorem 4.3]{Blecher-cstar}. 
\label{lem-Blecher-again}  Let $E$ be a $C^*$-$A$-$B$-correspon\-dence and let $F$ be a $C^*$-$B$-$C$-correspondence. The natural completely bounded map 
\[
E \otimes_{hB} F \longrightarrow E \otimes _{B} F
\]
from the Haagerup tensor product to the Kasparov tensor product is a completely isometric isomorphism. \qed
\end{lemma}

See also  \cite[Theorem 8.2.11]{BlM}.   But despite the lemma, and despite the theorem proved in the previous section, it is \emph{not} true that the two functors above are adjoint to one another.  Instead, the best result available is that there are natural isomorphisms
\begin{equation}\label{local_adjoint_equation}
 \Compact_{C^*_r(L)}(\Res X, Y) \cong \Compact_{C^*_r(G)}(X, \Ind Y) 
\end{equation}
between the spaces of \emph{compact} adjointable operators. See  \cite[Theorem 5.1]{CCH2}.

In contrast to all this, our operator module result, Theorem~\ref{G_operator_Frobenius_theorem},  is stronger and relies only on  the fact that $C^*_r(G)$ acts through compact operators on $C^*_r (G/N)$.  This is in turn an easy consequence of the geometry of $G$, involving no representation theory; {the essential point is that the homogeneous space $G/P$ is compact.}  

\subsection{SL(2,R)}
\label{sec-SL(2,R)1}  In order to explore the issues of the previous section a bit further, let us consider the special case of the group $SL(2,\R)$.

The general structure of the reduced $C^*$-algebra of a real reductive group is  summarized in \cite[Theorem 6.8]{CCH1}.  We won't repeat the general story  here, but instead we shall focus on $SL(2,\R)$ alone.  This example are also treated in \cite[Example 6.10]{CCH1}.\footnote{There is a long prior history of results on this topic (the reference \cite{CCH1} is certainly not a primary source) and we won't repeat that either, except to mention  \cite[Section 4]{BoMar}, where the reader can find a prior set of full details for  the  $SL(2,\R)$ calculation.}  

Up to conjugacy there is a unique nontrivial parabolic subgroup in $G=SL(2,\R)$, namely the group $P$ of upper triangular matrices, with Levi factor $L$ the diagonal matrices in $SL(2,\R)$.  The (necessarily one-dimensional) irreducible unitary representations of $L$ divide into two classes---the \emph{even} representations  where $\left [ \begin{smallmatrix} -1 & 0 \\ 0 & -1 \end{smallmatrix}\right ]$ acts as $1$, and the \emph{odd} representations where it acts as $-1$.  
There is accordingly a direct sum decomposition 
\[
C^*_r(L)\cong C^*_r(L)_{\text{even}} \oplus C^*_r(L)_{\text{odd}} 
\]
in which the even representations factor through the projection onto the even summand, and the odd representations factor through the projection onto the odd summand.  Both summands are isomorphic to $C_0(\R)$ as $C^*$-algebras.

Parabolically inducing the even and odd unitary representations of $L$, we obtain the even and odd \emph{principal series} representations of $G$.  Apart from these, among the irreducible unitary representations of $G$ that are weakly contained in the regular representation there are also  the \emph{discrete series} representations. Associated to this division of the representations  of $C^*_r (G)$ into three types there is a three-fold direct sum decomposition
\[
C^*_r (G)
\cong  
C^*_r(G)_{\text{discrete}}  \oplus C^*_r(G)_{\text{even}} \oplus C^*_r(G)_{\text{odd}} .
\]

Finally, there is a compatible direct sum decomposition 
\[
C^*_r (G/N) \cong C^*_r (G/N)_{\text{even}} \oplus C^*_r (G/N) _{\text{odd}}
\]
under which the reduced $C^*$-algebras of both $G$ and $L$ act on the even and odd parts through the projections onto their respective even and odd summands.

In what follows we shall concentrate on the \emph{even} summands.  The odd summands are similar, but a bit harder to describe in the case of $C^*_r (G)$.  However the situation as regards adjunctions is actually simpler and less interesting for the odd summands, and this is the reason that we shall concentrate on the even parts.  The discrete part of $C^*_r(G)$ plays no role at all, since it acts trivially on $C^*_r (G/N)$.

There is a $C^*$-algebra isomorphism
\[
C^*_r (G)_{\text{even}} \cong  C_0\bigl (\R, \Compact(H)\bigr )^{\mathbb Z/2\mathbb Z} 
\]
where $H$ is a separable infinite-dimensional Hilbert space, and  the two-element group $\mathbb Z/2\mathbb Z$ acts  on $\R$ by multiplication by $-1$, while it acts on $\Compact(H)$ trivially.
 
There is an isomorphism of Hilbert modules 
\[
C^*_r (G/N)_{\text{even}} \cong C_0(\R, H)
\]
under which 
\begin{enumerate}[\rm (a)]
\item The left action of $C^*_r (G)_{\text{even}}$ becomes  the obvious pointwise action under the isomorphisms given above.
\item The right action of $C^*_r (L)_{\text{even}}$ is by pointwise multiplication under the identification of $C^*_r(L)_{\text{even}}$ with $C_0(\R)$, and the inner product is the pointwise inner product.
 \item The $C^*_r(G)$-valued inner product on $C^*_r(G/N)_{\text{even}}^*$ takes values in the ideal $C^*_r (G)_{\text{even}}$, and is given by
 \[
  \langle f_1,f_2\rangle_{C^*_r (G)} =   \tfrac{1}{2} f_1   \otimes f_2 ^* + \tfrac{1}{2}  w(f _1)  \otimes w(f_2)^*  ,
  \]
  where $f_1,f_2\in C_0(\R, H)$, where 
  \[
  w(f)(x) = f(-x),
  \]
   and where the tensors on the right hand side are to be viewed as rank one adjointable  operators on  $C_0(\R, H)$.
\end{enumerate}
From all of this, and keeping in mind the obvious Morita equivalence
\[
 C_0\bigl (\R, \Compact(H)\bigr )^{\mathbb Z/2\mathbb Z}  \underset{\text{Morita}}\sim C_0(\R)^{\mathbb Z / 2 \mathbb Z},
 \]
{ we find that the problem of formulating an adjunction theorem for the
$C^*$-correspondence  $C^*_r(G/N)$ comes down to the same  for the data
\[ 
A  =  C_0(\R)^{\mathbb Z / 2 \mathbb Z}, \qquad B = C_0(\R), \qquad E  = C_0(\R),\]
with $E$ being regarded as a $C^*$-$A$-$B$-correspondence in the obvious way.} 
}

Frobenius reciprocity in the operator-module setting (Theorem \ref{G_operator_Frobenius_theorem}) reduces here to simple case considered in Example \ref{operator_Frobenius_example}: the unit 
\[
\eta \colon A\longrightarrow  E\otimes_{hB} E^ *
\]
 is the inclusion (the tensor product is canonically isomorphic to $B$ via the product), while the counit 
 \[
 \epsilon \colon E^*\otimes_{hA} E\longrightarrow  B
 \]
  is the product. 
  
  In contrast, the local adjunction isomorphism \eqref{local_adjoint_equation} in the Hilbert $C^*$-module setting is equivalent to the assertion that the conjugate operator space structure on $E^*$ coincides with  one induced by an $A$-valued inner product, namely the inner product
   \[ 
\langle f_1,f_2\rangle_A = \tfrac12  f_1^*f_2 +  \tfrac12  w (f_1^*f_2)   .
\]
The failure of the local adjunction isomorphism to extend to an isomorphism on all adjointable operators is a consequence of the fact that the counit 
$\epsilon$ defined above is a completely bounded map of $B$-bimodules, but not an adjointable map of Hilbert modules 
when the Haagerup tensor product is identified with Kasparov's internal tensor product using Lemma~\ref{lem-Blecher-again}.

\section{The Second Adjoint Theorem}\label{Bernstein_section}

For smooth representations of reductive $p$-adic groups, Bernstein made the remarkable discovery  that parabolic induction has not only a left adjoint, but also a right adjoint too, which is also given by parabolic restriction, but with respect to the \emph{opposite} parabolic subgroup (the transpose). See \cite{Bernstein-2-adjoint}, or, for an exposition, \cite[Chapter VI]{Renard}.

Bernstein's second adjoint theorem plays an important foundational role in the representation theory of $p$-adic groups, leading to a direct product decomposition of the category of smooth representations into component categories.   See for example  \cite[Chapter VI]{Renard} again.    Similar structure can be seen in the \emph{tempered} representation theory of both real and $p$-adic reductive groups, and one of the main  motivations for the work presented in \cite{CCH2,CCH1} was to obtain something similar to Bernstein's theorem in categories of representations related to the reduced group $C^*$-algebra. 

The local adjunction  isomorphism of \cite{CCH2} that we described in Section~\ref{sec-local-adj} is a partial solution. But it is not altogether satisfactory, since in the $p$-adic context Bernstein's theorem is a geometric foundation from which representation theory may be built up,\footnote{In this context see the recent article \cite{BezK} for a beuatiful, geometric approach to Bernstein's original theorem.} whereas our local adjunction theorem required an extensive acquaintance with tempered representation theory to formulate and prove.

So the question remains whether or not a suitable counterpart of Bernstein's second adjoint theorem can be developed in an operator-algebraic context.  We shall investigate this issue in detail elsewhere; our purpose here is to present two computations in the simple case of the group $SL(2,\R)$ that together indicate a possibly interesting role here for operator algebras and operator modules.

\subsection{Harish-Chandra's Schwartz space}  If $G$ is a real reductive group, as before, then its
 \emph{Harish-Chandra algebra}  is a Fr\'echet convolution algebra $\HC(G)$ of smooth, complex valued functions on $G$ that is perhaps easiest to present here as a distinguished subalgebra of    $C^*_r(G)$ that is closed under the holomorphic functional calculus.  
 
 The definition of $\HC(G)$ is a bit involved. Moreover  it   is not by any means obvious, even after one has mastered the definitions, that $\HC(G)$ is closed under convolution multiplication (see for example \cite[Section 7.1]{Wallach1} for the details). We shall avoid these difficulties here by using a Fourier-dual description of $\HC(G)$   that will suffice for our present limited purposes; see the next section.

In any case, we shall study the following module category. In the context of Fr\'echet spaces, in this section  and the next, the symbol $\otimes$ will denote the completed projective tensor product of Fr\'echet spaces.

\begin{definition}
\label{def-sfmod}
Let $\mathcal{A}$ be a Fr\'echet algebra (that is, a Fr\'echet space equipped with a (jointly) continuous and associative multiplication operation).
 A \emph{smooth Fr\'echet module} over   $\mathcal{A}$ is a Fr\'echet space $V$ which is equipped with a continuous  $\mathcal{A}$-module structure, such that the evaluation map 
\[ \mathcal{A}\otimes_{\mathcal{A}} V\to V \qquad a\otimes v\mapsto av\]
is an isomorphism.
\end{definition}

 \begin{remark}
 \label{rem-btp}
The tensor product $ \mathcal{A}\otimes_{\mathcal{A}} V$ used in the above definition is  the quotient of the completed projective tensor product $\mathcal{A}\otimes V $ by the closed subsapce generated by the balancing relators 
\[
a_1a_2 \otimes v - a_1 \otimes a_2 v
\]
with $a_1,a_2\in \mathcal{A}$ and $v\in V$. 
\end{remark}

\begin{definition}
\label{def-sfmod-cat}
We denote by  $\SFMod_{\mathcal{A}}$  the category of smooth Fr\'echet modules over $\mathcal{A}$, with continuous $\mathcal{A}$-linear maps as morphisms.
\end{definition}

If $\mathcal{E}$ is a smooth $\mathcal{A}$-$\mathcal{B}$-Fr\'echet bimodule, then the tensor product construction  in Remark~\ref{rem-btp} gives us a tensor product   functor 
\[
\mathcal{E} \colon \SFMod_{\mathcal{B}}\longrightarrow \SFMod_{\mathcal{A}}.
\]
We shall study parabolic induction from the perspective of such functors in the next section.  We should remark that if $\mathcal{A}$ is a Frechet algebra, then it is not necessarily true that the mutliplication map 
\[
\mathcal{A}\otimes_{\mathcal{A}} \mathcal{A} \longrightarrow \mathcal{A}
\]
is an isomorphism, but this \emph{is} true for the Harish-Chandra algebras that we shall be studying.

\subsection{The Harish-Chandra algebra of SL(2,R)}  We shall now specialize to $G=SL(2,\R)$ and is parabolic subgroup $P=LN$ of upper triangular matrices. 

The Harish-Chandra algebra for $L$ admits a decomposition
\[
\HC(L) = \HC(L)_{\text{even}} \oplus \HC(L)_{\text{odd}} ,
\]
that is compatible with the decomposition of the reduced group $C^*$-algebra.  Both  the even and odd summands are isomorphic as Fr\'echet algebras to the space  $\mathcal{S}(\R)$ of Schwartz functions on the line, with pointwise multiplication. 

Similarly   there is a decomposition
\[
\HC(G) = \HC(G)_{\text{discrete}} \oplus  \HC(G)_{\text{even}} \oplus  \HC(G)_{\text{odd}} 
\]
that is compatible with the decomposition of the reduced $C^*$-algebra  in Section~\ref{sec-SL(2,R)1}. 

Once again we shall concentrate on the even parts. There is an isomorphism
 \[
\HC(G)_{\text{even}} \cong \mathcal{S}\left(\R, \mathcal{K}(H)\right)^{\Z/2\Z}
\]
in which the algebra appearing on the right is as follows.  
\begin{enumerate}[\rm (a)]
\item The Hilbert space $H$ has a preferred orthonormal basis indexed by even integers (the Hilbert space carries an  $SO(2)$ representation, and the basis vectors are weight vectors).  
\item The right-hand algebra consists of continuous functions $f$ from $\R$ into the compact operators on $H$, invariant under the same $\Z/2\Z$ action as before.
\item If $p$ is any continuous seminorm on the space of  Schwartz functions on the line, and if $f_{ij}$ denotes the $ij$-matrix entry of $f$ with respect to the given orthonormal basis of $H$, then $p(f_{ij})$ is of rapid decay in $i$ and $j$. 
\end{enumerate}
Compare \cite{Arthur} and \cite[Chapter 8]{Varadarajan}.

Finally there is the bimodule $\HC(G/N)$, which  consists of suitable rapid decay functions on $G/N$, as in  \cite[Section 15.3]{Wallach2}. There is a decomposition 
\[
\HC(G/N) = \HC(G/N) _{\text{even}} \oplus \HC(G/N) _{\text{odd}} 
\]
as before, and there is an isomorphism
\[
\HC(G/N)_{\text{even}} \cong \mathcal{S}(\R, H) ,
\]
where on the right hand side are the functions $f\colon \R \to H$ whose component functions  $f_j$ are of rapid decay with respect to any Schwartz space seminorm, as in (c) above. 

Since there is again  a Morita equivalence 
\[
\mathcal{S}\left(\R, \mathcal{K}(H)\right)^{\Z/2\Z} \underset{\text{Morita}}\sim \S(\R)^{\Z/2\Z}
\]
(that is, an equivalence of  $\SFMod$ categories) we are finally reduced to studying adjunction theorems in the following Fr\'echet context:
\[ 
\mathcal{A}  =  \mathcal{S}(\R)^{\mathbb Z / 2 \mathbb Z}, \qquad 
\mathcal{B}  = \mathcal{S}(\R), \qquad
\mathcal{E}  = \mathcal{S}(\R),
\]
with $\mathcal{E}$ being assigned the structure of a smooth $\mathcal{A}$-$\mathcal{B}$-bimodule in the obvious way.

 So far this is of course an uninteresting reworking of the computations that we made in Section~\ref{sec-SL(2,R)1}.  And the situation with regard to Frobenius reciprocity is similarly predictable: 
if we define 
 \[
 \mathcal{F}   = \mathcal{S}(\R),
 \]
 with its obvious $\mathcal{B}$-$\mathcal{A}$-bimodule structure, then, exactly as before:

 \begin{theorem}
\label{thm-schw-frob}
 The bimodule maps
 \[ 
 \eta \colon \mathcal A \longrightarrow  \mathcal E\otimes_{\mathcal B} \mathcal F  \qquad \text{and} \qquad \epsilon\colon  \mathcal F\otimes_{\mathcal A} \mathcal E\longrightarrow  \mathcal B\]
 defined by
 \[ \eta(a_1a_2)=a_1\otimes a_2 \qquad \text{and} \qquad \epsilon (f\otimes e) = fe\]
 are the unit and counit of an adjunction.  \qed
\end{theorem}

But the situation with regard to ``Bernstein reciprocity,'' or the assertion that  $\mathcal{E}$ also has a \emph{right} adjoint, is much more interesting.  Surprisingly, in view of the fact that in most respects the Fr\'echet algebra $\HC(G)$ behaves much  like $C^*_r(G)$, there is a striking difference between the two regarding the second adjoint theorem, which in fact \emph{does} hold in the Harish-Chandra context.

We wish to define   a candidate unit map  
\[
 \mathcal{B} \longrightarrow \mathcal{F} \otimes _{\mathcal{A}} \mathcal{E} 
\]
as follows:
\begin{equation}
\label{eq-spectral-berns-unit}
b_1b_2\mapsto b_1x \otimes b_2  + b_1\otimes x b_2
\end{equation}
for $b_1,b_2\in \mathcal B$ (we are writing $x$ for the function $x\mapsto x$).  It is not immediately obvious that the formula is well-defined. But the following calculation shows that this is so:

\begin{lemma}\label{B-balanced-lemma}
The quantity in $ \mathcal{F} \otimes _{\mathcal{A}} \mathcal{E}$ described in \eqref{eq-spectral-berns-unit} depends only on the   product $b_1b_2\in \mathcal{B}$, and the formula defines a continuous $B$-bimodule homomorphism.
\end{lemma}

\begin{proof}
Let us first  show that if  $b\in \mathcal{B}$, then  
\[
 b_1bx \otimes b_2  + b_1b\otimes x b_2 =  b_1 x \otimes bb_2  + b_1 \otimes x bb_2 .
 \]
We can write 
\[
b = a_1 + a_2 x ,
\]
where $a_1,a_2\in \mathcal{A}$, and it suffices to consider separately the cases where $a_1=0$ and $a_2=0$.  The latter is easy, since the tensor products are over $\mathcal {A}$.  As for the former, we calculate that 
\[
\begin{aligned}
 b_1(a_2x)x \otimes b_2  + b_1(a_2x)\otimes x b_2 &=  b_1   \otimes a_2x^2b_2  + b_1x \otimes   a_2 xb_2 \\
 	&=  b_1x \otimes   a_2 xb_2 +  b_1   \otimes x a_2x   b_2  ,
		\end{aligned}
 \]
 as required (we used the fact that $x^2 a_2 \in \mathcal {A}$).  So the formula defines a continuous map 
 \[
 \mathcal{B}\otimes _{\mathcal{B}} \mathcal{B} \longrightarrow \mathcal{F} \otimes _{\mathcal{A}} \mathcal{E},
 \]
 and the lemma follows from the easily verified   fact that the multiplication map 
 \[
  \mathcal{B}\otimes _{\mathcal{B}} \mathcal{B} \longrightarrow \mathcal {B}
  \]
  is an isomorphism.
\end{proof}

\begin{remark}
\label{rem-c-function}
Bernstein constructed the unit map for his second adjunction using the geometry of the homogeneous space $G/N$, and in particular the fact that if $\overline P = L \overline N$ is the opposite parabolic subgroup, then the product map
\[
\overline{N} \times L \times N \longrightarrow G
\]
embeds the left hand side as an open subset of $G$.  See for example \cite[Section 3.1]{Bernstein-p-adic}.  It is   not immediately apparent, but the unit described here is essentially the same, and differs only in that  we have used the function $x\mapsto x$ in place of (the reciprocal of) Harish-Chandra's $c$-function from the theory of spherical functions.  The $c$-function arises when one calculates Bernstein's unit map from the spectral, or Fourier dual, perspective.
\end{remark}

We can now prove a counterpart of Bernstein's second adjoint theorem:
 
\begin{theorem}
\label{thm-spectral-2nd-adj}
The bimodule map given by the formula \eqref{eq-spectral-berns-unit} is the unit map for an adjunction 
\[
 \Hom_{\mathcal B}(Y, \mathcal{F}\otimes_{\mathcal{B}} X) \cong \Hom_{\mathcal A}(\mathcal{E}\otimes_{\mathcal{B}} Y, X).
\]
\end{theorem}

\begin{proof}
In order to prove the theorem we need to find a suitable counit map
\[
\mathcal{E} \otimes _{ \mathcal{B}} \mathcal{F}  \to \mathcal{A} .
\]
We shall use the formula
\begin{equation}
\label{eq-berns-spectral-counit}
e\otimes f \mapsto \frac1x  (ef)^{-} 
\end{equation}
in which the superscript  ``$-$'' on the right means that we take the \emph{odd part} of the function $ef\in \mathcal{B}$ (a superscript ``$+$'' will likewise denote the even part of a function).

The composition 
\[
\mathcal{E}\otimes _{\mathcal{B}} \mathcal{B} \longrightarrow \mathcal{E} \otimes _{\mathcal{B}} \mathcal{F} \otimes _{\mathcal {A}}\mathcal{E} \longrightarrow \mathcal{A} \otimes _{\mathcal{A} } \mathcal{E}
\longrightarrow \mathcal{E}
\]
is given by the formula
\[
\begin{aligned}
e \otimes  b_1b_2  
& \mapsto 
e \otimes   b_1x \otimes b_2 +e  \otimes  b_1 \otimes  x b_2 \\
& \mapsto 
\frac1x  (eb_1x)^{-}  \otimes  b_2  + \frac1x  (eb_1)^{-}   \otimes  xb_2 \\
& \mapsto 
   \frac 1x  (eb_1x)^{-} b_2  +    \frac 1x  (eb_1)^{-} xb_2   \\ 
   & = 
    (eb_1)^{+}b_2  +     { (eb_1)^{-}} b_2   \\ 
& = 
 e  b_1b_2,
\end{aligned}
\]
and this is the standard multiplication map, as required.
In addition  the  composition 
\[
\mathcal{B}\otimes_\mathcal{B}  \mathcal {F} \longrightarrow \mathcal {F}\otimes_{\mathcal{A}}  \mathcal{E} \otimes_\mathcal{B}\mathcal {F} \longrightarrow \mathcal {F} \otimes _{\mathcal{A}}\mathcal{A}
\longrightarrow \mathcal{F}
\]
is given by the formula 
\[
\begin{aligned}
b_1b_2\otimes  f & \mapsto b_1x \otimes  b_2 \otimes  f  + b_1\otimes  xb_2 \otimes  f  \\
& \mapsto 
b_1x \otimes \frac1x (b_2f)^{-}  + b_1 \otimes  \frac 1x   (fb_2x)^{-}  \\  
& \mapsto 
b_1{ (b_2f)^{-}}  + b_1  \frac 1x (fb_2x)^{-}  \\
& = b_1(b_2f)^{-} + b_1 (b_2 f )^+   ,
\end{aligned}
\]
which gives us the standard module multiplication map once again, as required. 
\end{proof}

\begin{remark}
In the present context of Harish-Chandra spaces, the bimodules $\HC(G/N)$ and $\HC(G/\overline N)$ are in fact isomorphic to one another, so it is not possible to detect the use of the opposite parabolic subgroup, except indirectly through the geometric role it plays in giving the formula for the unit map, as indicated in Remark~\ref{rem-c-function}.
\end{remark}

 \subsection{Bernstein's Theorem and Operator Spaces}
 \label{sec-op-alg-bernstein}
 
 In this final section we shall adapt the Schwartz algebra computations of the previous section to the context of operator algebras. 
 
The formula \eqref{eq-berns-spectral-counit} for the Bernstein counit  does not make sense for arbitrary continuous functions, and so does not make sense at the level of (reduced) group  $C^*$-algebras. We will show however that the Bernstein reciprocity theorem of the previous section can be recovered after replacing $C^*$-algebras with non-self-adjoint operator algebras.

 Given $f\in C_0(\R)$, we shall continue to use the notation   
 \[
 f = f^+ + f^-
 \]
 for the decomposition of $f$ into its  even and odd parts. We shall also denote by  \[
 w \colon C_0(\R) \longrightarrow C_0(\R)
 \]
 the involution  given by the formula
 \[
 w(f)(x)=f(-x)
 .
 \]
 Let us now fix a smooth function $c$ on the line (with a singularity at $0\in \R$) with the following properties:
 \begin{enumerate}[\rm (a)]
 \item $c$ is odd,
 \item $c(x) = 1/x$ for $x$ near $0\in \R$, and 
 \item $c(x) = 1$ for large positive $x$.
 \end{enumerate}
 The notation is supposed to call to mind Harish-Chandra's $c$-function, which is the ultimate source of the function $c(x) = 1/x$ that appears in the previous section; see Remark~\ref{rem-c-function}.  We are simplifying matters somewhat here by insisting that $1/c$ is a smooth function, bounded at infinity (in the natural construction of the unit map, involving the actual $c$-function from representation theory, the boundedness condition does not hold).   But this is a relatively minor issue; see Remark~\ref{rem-c-function2} below.

 \begin{definition} 
 We shall denote by  $B\subseteq  C_0(\R)$   the space of those functions $f\in C_0(\R)$  for which the product $c\cdot f^{-}$ extends to a continuous (and necessarily even) function on $\R$.  Equivalently $B$ consists of those functions in $C_0(\R)$ whose odd  part  is differentiable at $0\in \R$.
 \end{definition}

\begin{lemma}
The formula
 \[ 
 \delta (b) = c\cdot b^{-}
 \]
defines a $w$-twisted    derivation from $B$ into $C_0(\R)$, so that 
\[
\delta(b_1b_2) = \delta(b_1)b_2 + w(b_1) \delta(b_2).
\]
for all $b_1,b_2\in B$.  As a result the formula  
 \begin{equation}
 \label{eq-B-embedding}
 b\longmapsto \begin{bmatrix} b & 0 \\ \delta(b) & w(b) \end{bmatrix}
 \end{equation}
 defines an algebra embedding of $B$ into the algebra $2\times 2$ matrices over $C_0(\R)$. \qed
\end{lemma}

\begin{lemma}
The image of $B$ in $M_2(C_0(\R))$ under the embedding \eqref{eq-B-embedding} is a norm-closed subalgebra.  \qed
 \end{lemma}

We shall equip the algebra $B$ with the operator algebra structure it receives from the embedding \eqref{eq-B-embedding}. 
In addition, let  $A $ be the $C^*$-algebra of even  functions in $C_0(\R)$.  It embeds completely isometrically into $B$, and 
\[
\overline{AB} = B = \overline{BA}
\]
(indeed the closures are superfluous).

Let $E=B$, considered as an operator $A$-$B$-bimodule, and let $F=B$ considered as an operator $B$-$A$-bimodule. Frobenius reciprocity, or the assertion that the tensor product functor 
\[
F\colon  \OMod_A\longrightarrow  \OMod_B 
\]
is left-adjoint to the tensor product functor
\[
E 
\colon \OMod_B\longrightarrow  \OMod_A 
 \]
holds as in Example \ref{operator_Frobenius_example}. But in addition these modules   satisfy the following version of Bernstein reciprocity:

\begin{theorem}\label{operator_Bernstein_theorem}
 The tensor product functor $E$ is left-adjoint to the tensor product functor $F$: there is a natural isomorphism
 \[ \CB_B(Y, F\otimes _{hA} X) \cong \CB_A( E \otimes_{hB}  Y, X).\]
\end{theorem}

\begin{proof}
We want to define a unit   map  
\begin{equation*}
\eta\colon  B\longrightarrow  F\otimes_{hA} E 
\end{equation*}
by the formula 
\begin{equation}\label{eq-c-fn-unit}
 b_1b_2\longmapsto \frac {b_1}c\otimes b_2 + b_1\otimes\frac {b_2} c .
\end{equation}
{ The formula gives a well-defined map by the argument of Lemma \ref{B-balanced-lemma}, which applies here because every element of $B$ is of the form $a_1+a_2/c$. The map is completely bounded because the function $\frac{1}{c}$ is a bounded multiplier of $C_0(\R)$. Clearly $\eta$ is a $B$-bimodule map.
}

In addition, define a counit map 
\[ 
\epsilon \colon  E\otimes_{hB} F \longrightarrow  A
\]
by the formula
\[ 
  e \otimes f \longmapsto  c (e  f)^{-}
.
\]
{This is certainly an $A$-bimodule map.}  It can  be viewed as the composition
 \[
 \xymatrix{
 B\otimes _{hB} B \ar[r]  & B \ar[r]^\delta  & A
 }
 \]
 in which the first map is just the  multiplication map on $B$, which is completely bounded.  As for $\delta$, it is the restriction to $B$ of the completely bounded map 
 \[
 T\longmapsto \begin{bmatrix}  0 & 1\end{bmatrix}  T  \begin{bmatrix} 1 \\ 0\end{bmatrix} 
 \]
from  $M_2(C_0(\R))$ to $C_0(\R)$, and so it too is completely bounded.

The verification, now, that $\eta$ and $\epsilon$ are the unit and counit of an adjunction is   exactly as in the proof of Theorem \ref{thm-spectral-2nd-adj}.
\end{proof}

 \begin{remark}
 \label{rem-c-function2}
 If the function $1/c$ was unbounded (as it would be if we were to use the natural, representation-theoretic $c$-function), then our formula   \eqref{eq-c-fn-unit} for the unit map $\eta$ would give an unbounded, densely-defined $B$-bimodule map.  Its domain, an ideal in ${B}$, would be an operator algebra in its own right, and we could repeat the above argument with this algebra in place of the original ${B}$. 
 \end{remark}


\end{document}